\documentclass[11pt]{article}

\usepackage{fullpage}
\usepackage{amsmath, amssymb, amsthm}
\usepackage{abstract}

\theoremstyle{plain}
\newtheorem{theorem}{Theorem}
\newtheorem{lemma}[theorem]{Lemma}
\newtheorem{corollary}[theorem]{Corollary}
\theoremstyle{remark}
\newtheorem{question}{Question}

\title{Answer to a question of Alon and Lubetzky\\ about the ultimate categorical independence ratio}

\author{{\Large \'Agnes T\'oth}\thanks{The work reported in the paper has been developed in the framework of the project ``Talent care and cultivation in the scientific workshops of BME'' project. This project is supported by the grant T\'AMOP - 4.2.2.B-10/1--2010-0009}\\
Department of Computer Science and Information Theory,\\ Budapest University of Technology and Economics, Hungary\\ {\tt tothagi@cs.bme.hu}}

\date{}

\begin{document}

\maketitle

\begin{abstract}
Brown, Nowakowski and Rall defined the ultimate categorical independence ratio of a graph~$G$ as~$A(G)=\lim\limits_{k\to \infty} i(G^{\times k})$, where $i(G)=\frac{\alpha (G)}{|V(G)|}$ denotes the independence ratio of a graph $G$, and $G^{\times k}$~is the $k$th categorical power of $G$.
Let $a(G)=\max\{\frac{|U|}{|U|+|N_G(U)|}:\textnormal{$U$ is an independent set of $G$}\}$, where $N_G(U)$ is the neighborhood of $U$ in $G$.
In this paper we answer a question of Alon and Lubetzky, namely we prove that if $a(G)\le \frac{1}{2}$ then $A(G)=a(G)$, and if $a(G)>\frac{1}{2}$ then $A(G)=1$.
We also discuss some other open problems related to $A(G)$ which are immediately settled by this result.
\end{abstract}

\section{Introduction}\label{sec1}

The \emph{independence ratio} of a graph $G$ is defined as $i(G)=\frac{\alpha (G)}{|V(G)|},$ that is, as the ratio of the independence number and the number of vertices.
For two graphs $G$ and $H$, their \emph{categorical product} (also called as direct or tensor product) $G\times H$ is defined on the vertex set $V(G\times H)=V(G)\times V(H)$ with edge set $E(G\times H)=\{\{(x_1, y_1),(x_2, y_2)\}\, :\, \{x_1,x_2\}\in E(G) \textnormal{ and } \{y_1,y_2\}\in E(H)\}$. The $k$th categorical power $G^{\times k}$ is the $k$-fold categorical product of $G$. The \emph{ultimate categorical independence ratio} of a graph $G$ is defined as
\[A(G)=\lim\limits_{k\to \infty} i(G^{\times k}).\]
This parameter was introduced by Brown, Nowakowski and Rall in \cite{bnr} where they proved that for any independent set $U$ of $G$ the inequality $A(G)\ge \frac{|U|}{|U|+|N_G(U)|}$ holds, where $N_G(U)$ denotes the neighborhood of $U$ in $G$. Furthermore, they showed that $A(G)>\frac{1}{2}$ implies $A(G)=1$.\\

\noindent Motivated by these results, Alon and Lubetzky \cite{al} defined the parameters $a(G)$ and $a^*(G)$ as follows
\[a(G)=\max\limits_{\textnormal{$U$ is indep. set of $G$}} \frac{|U|}{|U|+|N_G(U)|}
\hspace*{20pt}\textnormal{ and }\hspace*{20pt}
a^*(G)=\begin{cases}
a(G) & \textnormal{ if } a(G)\le \frac{1}{2}\\
1 & \textnormal{ if } a(G)> \frac{1}{2}
\end{cases},\]
and they proposed the following two questions.

\begin{question}[\cite{al}]\label{strong}
Does every graph $G$ satisfy $A(G)=a^*(G)$?
Or, equivalently, does every graph $G$ satisfy $a^*(G^{\times 2})=a^*(G)$?
\end{question}

\begin{question}[\cite{al}]\label{weak}
Does the inequality $i(G\times H)\le \max\{a^*(G),a^*(H)\}$ hold for every two graphs $G$ and $H$?
\end{question}

\noindent The above results from \cite{bnr} give us the inequality $A(G)\ge a^*(G)$. One can easily see the equivalence between the two forms of Question \ref{strong}, moreover it is not hard to show that an affirmative answer to Question \ref{strong} would imply the same for Question \ref{weak} (see \cite{al}).

Following \cite {bnr} a graph $G$ is called self-universal if $A(G)=i(G)$. As a consequence, the equality $A(G)=a^*(G)$ in Question \ref{strong} is also satisfied for these graphs according to the chain inequality $i(G)\le a(G)\le a^*(G)\le A(G)$. Regular bipartite graphs, cliques and Cayley graphs of Abelian groups belong to this class \cite{bnr}. In \cite{my} the author proved that a complete multipartite graph is self-universal, except for the case when $a(G)>\frac{1}{2}$, therefore the equality $A(G)=a^*(G)$ is also verified for this class of graphs. (In the latter case $A(G)=a^*(G)=1$.) In \cite{al} it is shown that the graphs which are disjoint union of cycles and complete graphs satisfy the inequality in Question \ref{weak}.

\medskip

In this paper we answer Question \ref{strong} affirmatively. Thereby a positive answer also for Question \ref{weak} is obtained. Moreover it solves some other open problems related to $A(G)$.
In the proofs we exploit an idea of Zhu~\cite{z} that he used on the way when proving the fractional version of Hedetniemi's conjecture. In Section~\ref{sec2} this tool is presented. Then, in Section~\ref{sec3}, first we prove the inequality 
\[i(G\times H)\le \max\{a(G),a(H)\} \textnormal{, \, for every two graphs $G$ and $H$},\]
and give a positive answer to Question \ref{weak} (using $a(G)\le a^*(G)$). 
Afterwards we prove that
\[a(G\times H)\le\max\{a(G),a(H)\} \textnormal{, \, provided that $a(G)\le \frac 12$ or $a(H)\le \frac 12$},\]
and from this result we conclude the affirmative answer to Question \ref{strong}. (If $a(G)>\frac{1}{2}$ then $a^*(G^{\times 2})=a^*(G)=1$. Otherwise applying the above result for $G=H$ we get $a(G^{\times 2})\le a(G)$, while the reverse inequality clearly holds for every~$G$. Thus we have $a^*(G^{\times 2})=a^*(G)$ for every graph~$G$.)
Finally, in Section~\ref{sec4}, we discuss further open problems which are solved by our result. For instance, we get a proof for the conjecture of Brown, Nowakowski and Rall, stating that $A(G\cup H)=\max\{A(G), A(H)\}$, where $G\cup H$ is the disjoint union of $G$ and $H$.

\section{Zhu's lemma}\label{sec2}

Recently Zhu \cite{z} proved the fractional version of Hedetniemi's conjecture, that is, he showed that for every graph $G$ and $H$ we have $\chi_f(G\times H)=\min\{\chi_f(G),\chi_f(H)\}$, where $\chi_f(G)$ denotes the fractional chromatic number of the graph $G$. During the proof he showed the following result on the independent sets of categorical product of graphs. This will be the key idea also in our case.

Let $U$ be an independent set of $G\times H$. Zhu considered the partition $U$ into $U=A\uplus B$, where 
\begin{equation}\label{part}
\begin{array}{l}
A=\{(x,y)\in U\, :\, \nexists (x',y)\in U \textrm{ s.t. } \{x,x'\}\in E(G)\},\\
B=\{(x,y)\in U\, :\, \exists (x',y)\in U \textrm{ s.t. } \{x,x'\}\in E(G)\}.
\end{array}
\end{equation}

\noindent In the sequel, we keep using the following notations for any $Z\subseteq V(G\times H)$.\\
For any $y\in V(H)$, let
\[Z(y)=\{x\in V(G)\, :\, (x,y)\in Z\}.\]
Similarly, for any $x\in V(G)$, let
\[Z(x)=\{y\in V(H)\, :\, (x,y)\in Z\}.\]
And, let
\[N^G(Z)=\{(x,y)\in V(G\times H)\, :\, x\in N_G(Z(y))\}.\]
In words, $N^G(Z)$ means that we decompose $Z$ into sections corresponding to the elements of $V(H)$, and in each section we pick those points which are neighbors of the elements of $Z(y)$ in the graph~$G$.
Similarly, let \[N^H(Z)=\{(x,y)\in V(G\times H)\, :\, y\in N_H(Z(x))\}.\]
Keep in mind, that $Z(y)\subseteq V(G)$ and $Z(x)\subseteq V(H)$, while $N^G(Z), N^H(Z)\subseteq V(G\times H)$.

\begin{lemma}[\cite{z}]\label{zhu}
The following holds:\\
(1) For every $y\in V(H)$, $A(y)$ is an independent set of $G$. For every $x\in V(G)$, $B(x)$ is an independent set of $H$.\\
(2) $A$, $B$, $N^G(A)$ and $N^H(B)$ are pairwise disjoint subsets of $V(G\times H)$.
\end{lemma}

For the sake of completeness we prove this lemma.

\begin{proof}
$A(y)$ is independent for every $y\in V(H)$ by definition. If for any $x\in V(G)$ the set $B(x)$ is not independent in $H$, that is $\exists y, y'\in B(x)$, $\{y,y'\}\in E(H)$, then from $(x,y')\in B$ we get that $\exists (x',y')\in U$, $\{x,x'\}\in E(G)$. This is a contradiction, because $(x,y)\in B$ and $(x',y')\in U$ were two adjacent elements of the independent set $U$.

Now we show the second part of the lemma. By definition $A\cap B=\emptyset$. The first part of the lemma implies that the pair $(A, N^G(A))$ is also disjoint, as well as the pair $(B, N^H(B))$.\\
\newpage

\noindent If $(x,y)\in A\cap N^H(B)$ then (by the definition of $N^H(B)$) $\exists (x,y')\in B$, $\{y,y'\}\in E(H)$, and so (by the definition of $B$) $\exists(x',y')\in U$, $\{x,x'\}\in E(G)$, which is a contradiction: $(x,y)\in A$ and $(x',y')\in U$ are adjacent vertices in the independent set $U$.
Similarly, if $(x,y)\in N^G(A)\cap N^H(B)$ then (by the definition of $N^G(A)$) $\exists (x',y)\in A\subseteq U$, $\{x,x'\}\in E(G)$ while (by the definition of $N^H(B)$) $\exists (x,y')\in B\subseteq U$, $\{y,y'\}\in E(H)$, which contradicts to the independence of $U$.
Finally, $(x,y)\in B\cap N^G(A)$ implies that $\exists (x',y)\in A$, $\{x,x'\}\in E(G)$ (by the definition of $N^G(A)$), which is in contradiction with the definition of $A$: there should not be an $(x,y)\in B\subseteq U$ satisfying $\{x,x'\}\in E(G)$.
\end{proof}

\section{Proofs}\label{sec3}

In this section we prove the statements mentioned in the Introduction. In Subsection~\ref{subsec31} we give an upper bound for $i(G\times H)$ in terms of $a(G)$ and $a(H)$. In Subsection~\ref{subsec32} we prove that the same upper bound holds also for $a(G\times H)$ provided that $a(G)\le\frac{1}{2}$ or $a(H)\le\frac{1}{2}$. Thereby we obtain that $A(G)=a^*(G)$ for every graph $G$.

\subsection{Upper bound for $i(G\times H)$}\label{subsec31}

As a simple consequence of Zhu's result the following inequality is obtained.

\begin{theorem}\label{weakthm}
For every two graphs $G$ and $H$ we have
\[i(G\times H)\le \max\{a(G),a(H)\}.\]
\end{theorem}

\begin{proof}
Let $U$ be a maximum-size independent set of $G\times H$, then we have
\begin{equation}\label{p1}
i(G\times H) = \frac{\alpha (G\times H)}{|V(G\times H)|} = \frac{|U|}{|V(G\times H)|}.
\end{equation}
We partition $U$ into $U=A\uplus B$ according to (\ref{part}). We also use the notations $A(y)$ for every $y\in V(H)$, $B(x)$ for every $x\in V(G)$, and $N^G(A)$, $N^H(B)$ defined in the previous section.\\
It is clear that $|U|=|A|+|B|$. From the second part of Lemma~\ref{zhu} we have that $|A|+|B|+|N^G(A)|+|N^H(B)|\le |V(G\times H)|$. Observe that $|N^G(A)|=\sum_{y\in V(H)}|N_G(A(y))|$ and $|N^H(B)|=\sum_{x\in V(G)}|N_H(B(x))|$. Hence we get 
\begin{multline}\label{p2}
\frac{|U|}{|V(G\times H)|} \le \frac{|A|+|B|}{|A|+|B|+|N^G(A)|+|N^H(B)|} = \\
=\frac{\sum_{y\in V(H)}|A(y)|+\sum_{x\in V(G)}|B(x)|}{\sum_{y\in V(H)}(|A(y)|+|N_G(A(y))|) + \sum_{x\in V(G)}(|B(x)|+|N_G(B(x))|)}.
\end{multline}
From the first part of Lemma~\ref{zhu} and by the definition of $a(G)$ and $a(H)$ we have $\frac{|A(y)|}{|A(y)|+|N_G(A(y))|}\le a(G)$ for every $y\in V(H)$, and $\frac{|B(x)|}{|B(x)|+|N_H(B(x))|}\le a(H)$ for every $x\in V(G)$, respectively.\\
Using the fact that if $\frac{t_1}{s_1}\le r$ and $\frac{t_2}{s_2}\le r$ then $\frac{t_1+t_2}{s_1+s_2}\le r$, this yields
\begin{equation}\label{p3}
\frac{\sum_{y\in V(H)}|A(y)|+\sum_{x\in V(G)}|B(x)|}{\sum_{y\in V(H)}(|A(y)|+|N_G(A(y))|) + \sum_{x\in V(G)}(|B(x)|+|N_H(B(x))|)} \le\\
\max\{a(G), a(H)\}.
\end{equation}
The inequalities (\ref{p1}), (\ref{p2}) and (\ref{p3}) together give us the stated inequality,
\[i(G\times H) \le \max\{a(G), a(H)\}.\]
\end{proof}

\bigskip

As we stated in the Introduction, from Theorem \ref{weakthm} it follows that the answer to Question \ref{weak} is positive.

\subsection{Answer to Question \ref{strong}}\label{subsec32}

In this subsection we answer Question \ref{strong} affirmatively. To show that $a^*(G^{\times 2})=a^*(G)$ holds for every graph $G$ it is enough to prove that $a(G^{\times 2})\le a(G)$ for every graph $G$ with $a(G)\le\frac{1}{2}$. Because if $a(G)>\frac{1}{2}$ then $a^*(G^{\times 2})=a^*(G)=1$, in addition every $G$ satisfies $a(G^{\times 2})\ge a(G)$. The condition $a(G)\le\frac{1}{2}$ is necessary, since otherwise $A(G)=1$ therefore $i(G^{\times k})$ and $a(G^{\times k})$ as well can be arbitrary close to $1$ for sufficiently large $k$.
A bit more general, we prove the following theorem.

\begin{theorem}\label{strongthm}
If $a(G)\le \frac 12$ or $a(H)\le \frac 12$ then
\[a(G\times H)\le\max\{a(G),a(H)\}.\]
\end{theorem}

\begin{proof}
We will show that for every independent set $U$ of $G\times H$ we have
\[\frac{|U|}{|U|+|N_{G\times H}(U)|}\le \max\{a(G),a(H)\}.\]
First, let $\hat{A}$, $\hat{B}$ and $C$ be the following subsets of $U$.
\begin{gather*}
\hat{A}=\{(x,y)\in U\, :\, \nexists (x',y)\in U \textrm{ s.t. } \{x,x'\}\in E(G), \textrm{ but } \exists (x,y')\in U \textrm{ s.t. } \{y,y'\}\in E(H)\},\\
\hat{B}=\{(x,y)\in U\, :\, \nexists (x,y')\in U \textrm{ s.t. } \{y,y'\}\in E(H), \textrm{ but } \exists (x',y)\in U \textrm{ s.t. } \{x,x'\}\in E(G)\},\\
C=\{(x,y)\in U\, :\, \nexists (x',y)\in U \textrm{ s.t. } \{x,x'\}\in E(G), \textrm{ and } \nexists (x,y')\in U \textrm{ s.t. } \{y,y'\}\in E(H)\}.
\end{gather*}
It is clear that $\hat{A}$, $\hat{B}$ and $C$ are pairwise disjoint. In addition, there is no $(x,y)\in U$ for which $\exists (x',y),(x,y')$ in $U$ such that $\{x,x'\}\in E(G)$ and $\{y,y'\}\in E(H)$, because $\{(x',y),(x,y')\}\in E(G\times H)$ and $U$ is an independent set. Hence {\bf ${\bf U}$ is partitioned into $\bf U=\hat{A}\uplus \hat{B}\uplus C$}.
(The connection with the partition of Zhu defined in (\ref{part}) is clearly the following, $A=\hat{A}\uplus C$ and $B=\hat{B}$.)

Observe that the definition of $a(G)$ can be rewritten as follows
$$\min\left\{\frac{|N_G(U)|}{|U|}\, :\, \textnormal{U is independent in $G$}\right\} = \frac{1-a(G)}{a(G)}.$$
Set $b(G)=\frac{1-a(G)}{a(G)}$. It it enough to prove that $\bf |N_{G\times H}(U)|\ge {\mathrm\mathbf min}\{b(G),b(H)\}|U|$. We shall give a lower bound for $|N_{G\times H}(U)|$ in two steps. 
\medskip

In the \textbf{first step} we consider the elements of $\hat{A}$ and $C$ for every $y\in V(H)$.
By definition $(\hat{A}\cup C)(y)$ is independent in $G$ for every $y\in V(H)$, therefore $|N_G((\hat{A}\cup C)(y))|\ge b(G)|(\hat{A}\cup C)(y)|$.\\
We partition $N^G(\hat{A}\cup C)$ into two parts, let 
$$N_1=N^G(\hat{A}\cup C)\cap N_{G\times H}(U)\textnormal{ \;\,\  and \;\,\ }M=N^G(\hat{A}\cup C)\setminus N_{G\times H}(U).$$ (It is easy to see that $N^G(\hat{A})\subseteq N_{G\times H}(U)$. However $N^G(C)\subseteq N_{G\times H}(U)$ is not necessarily true, that is why we make this partition.)
Thus for $N_1\subseteq N_{G\times H}(U)$ we have
\begin{equation}\label{n1}
|N_1|\ge b(G) \big(|\hat{A}|+|C|\big)-|M|.
\end{equation}

\medskip

In the \textbf{second step} we consider the elements of $\hat{B}$ and $M$ for every $x\in V(G)$. By the definition of $\hat{A}$ and $C$, $\hat{B}(x)$ and $M(x)$ are disjoint. Indeed, if $(x,y)\in M\subseteq N^G(\hat{A}\cup C)$ then $\exists (x',y)\in \hat{A}\cup C, \{x,x'\}\in E(G)$ and so $(x,y)$ cannot be in $\hat{B}\subseteq U$.

We claim that $(\hat{B}\cup M)(x)$ is independent in $V(H)$. Clearly, $\hat{B}(x)$ is independent. Furthermore, if $y,y'\in M(x), \{y,y'\}\in E(H)$ then
from $(x,y)\in M$ we get that $\exists (x',y)\in \hat{A}\cup C, \{x,x'\}\in E(G)$, hence $(x,y')\in M$ is a neighbor of $(x',y)\in U$ which contradicts that $M\cap N_{G\times H}(U)=\emptyset$. Similarly if $y\in \hat{B}(x), y'\in M(x), \{y,y'\}\in E(H)$ then from $(x,y)\in \hat{B}$ it follows that $\exists (x',y)\in U, \{x,x'\}\in E(G)$, but again, as $(x,y')\in M$ is a neighbor of $(x',y)\in U$ it is in a contradiction with the definition of $M$.
Therefore $|N_H((\hat{B}\cup M)(x))|\ge b(G) |(\hat{B}\cup M)(x)|$. Let
$$N_2=N^H(\hat{B}\cup M).$$ Considering the sum for all $x\in V(G)$ we obtain 
\begin{equation}\label{n2}
|N_2|\ge b(H) \big(|\hat{B}|+|M|\big).
\end{equation}

We show that $N_2\subseteq N_{G\times H}(U)$. On the one hand, if $y\in \hat{B}(x)$ and $y'$ is a neighbor of $y$ in $H$, and so $(x,y')\in N^H(\hat{B})$ then by the definition of $\hat{B}$, $\exists (x',y)\in U, \{x,x'\}\in E(G)$, hence $(x,y')$ is a neighbor of $(x',y)\in U$, that is, $(x,y')\in N_{G\times H}(U)$. On the other hand, if $y\in M(x)$ and $y'$ is a neighbor of $y$ in $H$, and so $(x,y')\in N^H(M)$ then by the definition of~$M$, $\exists (x',y)\in \hat{A}\cup C, \{x,x'\}\in E(G)$, therefore $\{(x',y),(x,y')\}\in E(G\times H)$, thus $(x,y')\in N_{G\times H}(U)$.

\medskip

Next we prove that the neighborhood sets gotten in the two steps, {\bf $\bf N_1$ and $\bf N_2$ are disjoint}. Suppose indirectly, that $(x,y)\in N_1\cap N_2$. Then $(x,y)\in N_1$ implies that $\exists (x',y)\in \hat{A}\cup C, \{x,x'\}\in E(G)$. While from $(x,y)\in N_2$ we get that $\exists (x,y')\in \hat{B}$ or $\exists (x,y')\in M$ satisfying $\{y,y'\}\in E(H)$. It is a contradiction since $(x',y)$ and $(x,y')$ are adjacent in $G\times H$, but no edge can go between $\hat{A}\cup C$ and $\hat{B}\cup M$ by the independence of $U$ and the definition of $M$.
As $N_1, N_2\subseteq N_{G\times H}(U)$ this yields
\begin{equation}\label{n1n2}
|N_{G\times H}(U)|\ge |N_1|+|N_2|.
\end{equation}

\medskip

From (\ref{n1}), (\ref{n2}) and (\ref{n1n2}) we obtain that
$$|N_{G\times H}(U)|\ge |N_1|+|N_2|\ge \bigg(b(G) \big(|\hat{A}|+|C|\big)-|M|\bigg) + \bigg(b(H) \big(|\hat{B}|+|M|\big)\bigg).$$
{\bf If $\bf a(H)\le \frac{1}{2}$}, that is $b(H)\ge 1$, then
\begin{multline*}
\bigg(b(G) \big(|\hat{A}|+|C|\big)-|M|\bigg) + \bigg(b(H) \big(|\hat{B}|+|M|\big)\bigg) \ge \\
\ge \min\{b(G),b(H)\}\bigg(|\hat{A}|+|\hat{B}|+|C|\bigg) + \big(b(H)-1\big)|M| \ge \min\{b(G),b(H)\}|U|.
\end{multline*}
Combining the latter two inequalities we obtain $|N_{G\times H}(U)|\ge \min\{b(G),b(H)\}|U|$, as desired.

\medskip

\noindent {\bf If $\bf a(G)\le \frac{1}{2}$} (and $a(H)>\frac{1}{2}$) we can change the role of $G$ and $H$ to get the same lower bound for $|N_{G\times H}(U)|$, or we can argue as follows. We distinguish two cases. First, suppose $|\hat{A}|+|C|-\frac{|M|}{b(G)}\ge 0$. By using $b(G)\ge 1$ this gives
\begin{multline*}
\bigg(b(G) \big(|\hat{A}|+|C|\big)-|M|\bigg) + \bigg(b(H) \big(|\hat{B}|+|M|\big)\bigg) = b(G) \left(|\hat{A}|+|C|-\frac{|M|}{b(G)}\right) + b(H) \big(|\hat{B}|+|M|\big) \ge \\
\ge \min\{b(G),b(H)\}\bigg(|\hat{A}|+|\hat{B}|+|C|+|M|\big(1-\frac{1}{b(G)}\big)\bigg) \ge \min\{b(G),b(H)\}|U|,
\end{multline*}
finishing the inequality chain.
While from $|\hat{A}|+|C|-\frac{|M|}{b(G)}<0$ and $b(G)\ge 1$ it follows $|\hat{A}|+|C|<|M|$, hence we have
$$|N_{G\times H}(U)|\ge |N_2|\ge b(H) \big(|\hat{B}|+|M| \big)\ge \min\{b(G),b(H)\}|U|.$$

\medskip

Consequently, $|N_{G\times H}(U)|\ge \min\left\{\frac{1-a(G)}{a(G)},\frac{1-a(H)}{a(H)}\right\}|U|$ in both cases, that is $\frac{|U|}{|U|+|N_{G\times H}(U)|}\le \max\{a(G),a(H)\}$, this completes the proof.
\end{proof}

\bigskip

We mentioned in the Introduction that the two forms of Question \ref{strong} are equivalent. Hence from the equality $a^*(G^{\times 2})=a^*(G)$ for every graph $G$ we obtain the following corollary. (Indeed, suppose on the contrary that $G$ is a graph with $a^*(G)<A(G)$ then $\exists k$ such that $a^*(G)<i(G^{\times k})\le a^*(G^{\times k})$, and as the sequence $\{a^*(G^{\times \ell })\}_{\ell =1}^{\infty}$ is monotone increasing, it follows that $\exists m$ for which $a^*(G^{\times m})<a^*(G^{\times 2m})$, giving a contradiction.)

\begin{corollary}\label{cor}
For every graph $G$ we have $A(G)=a^*(G)$.
\end{corollary}

\smallskip

\section{Further consequences}\label{sec4}

Brown, Nowakowski and Rall in \cite{bnr} asked whether $A(G\cup H) = \max\{A(G), A(H)\}$, where $G\cup H$ denotes the disjoint union of $G$ and $H$. From Corollary \ref{cor} we immediately receive this equality since the analogue statement, $a^*(G\cup H) = \max\{a^*(G), a^*(H)\}$ is straightforward. In \cite{al} it is shown that $A(G\cup H)=A(G\times H)$, therefore we have
\[A(G\cup H)=A(G\times H)=\max\{A(G),A(H)\}, \textnormal{\, for every graph $G$ and $H$}.\]

The authors of \cite{bnr} also addressed the question whether $A(G)$ is computable, and if so what is its complexity. They showed that if $G$ is bipartite then $A(G)=\frac{1}{2}$ if $G$ has a perfect matching, and $A(G)=1$ otherwise. Hence for bipartite graphs $A(G)$ can be determined in polynomial time. 
Moreover, it is proven in \cite{al} that $a(G)\le \frac{1}{2}$ if and only if $G$ contains a fractional perfect matching. Therefore given an input graph $G$, determining whether $A(G)=1$ or $A(G)\le \frac{1}{2}$ can be done in polynomial time. They also mentioned that deciding whether $a(G)>t$ for a given graph $G$ and a given value $t$, is NP-complete. From Corollary \ref{cor} we can conclude that $A(G)$ can be calculated, and the problem of deciding whether $A(G)>t$ is NP-complete too.

Although any rational number in $(0,\frac 12]\cup\{1\}$ is the ultimate categorical independence ratio for some graph $G$, as it is showed \cite{bnr}. Here we remark that we obtained that $A(G)$ cannot be irrational, solving another problem mentioned in \cite{bnr}.

\end{document}